\documentclass[11pt]{amsart}
\usepackage[latin1]{inputenc}
\usepackage{mathrsfs}
\usepackage{amsmath}
\usepackage{amsfonts}
\usepackage{amssymb}
\usepackage{graphicx}
\usepackage{color}
\usepackage{epstopdf} 
\usepackage{epsfig}
\usepackage{tikz}
\usepackage[normalem]{ulem}

\usepackage{a4}
\usepackage[all]{xy}

\newtheorem{thm}{Theorem}[section]

\newtheorem{lem}[thm]{Lemma}

\newtheorem{wn}[thm]{Corollary}
\newtheorem{obs}[thm]{Observation}

\newcommand{\gwc}{\gamma_{\rm wcon}}
\newcommand{\gc}{\gamma_c}

\usepackage{color}

\begin{document}
\title[On the connected and weakly convex domination numbers]{On the connected and weakly convex domination numbers}
\author[Dettlaff]{Magda Dettlaff}
\address{Gda{\'n}sk University of Technology, Gda{\'n}sk, Poland}
 {\email{mdettlaff@mif.pg.gda.pl}
\author[Lema\'nska]{Magdalena Lema\'nska}
\address{Gda{\'n}sk University of Technology, Gda{\'n}sk, Poland}
 \email{magda@mif.pg.gda.pl}
  \author[Osula]{Dorota Osula}
\address{Faculty of Electronics, Telecommunications and Informatics, Gda{\'n}sk University of Technology, Gda{\'n}sk, Poland}
 {\email{dorurban@student.pg.edu.pl}

\author[Souto-Salorio]{Mar\'ia Jos\'e Souto-Salorio}
\address{Facultade de Informatica, Campus de Elvi\~{n}a,
Universidade da Coru\~{n}a, CP 15071, A Coru\~{n}a, Espa\~{n}a}
{\email{maria.souto.salorio@udc.es}
\date{\today}


\thanks{Supported by TIN2017-85160-C2-1-R   projects    funds (Spain) and National Science Centre grant number 2015/17/B/ST6/01887 (Poland).}


\maketitle

 \begin{abstract}
In this paper we study relations between connected and weakly convex domination numbers.  We show that in general the difference between these numbers can be arbitrarily large and we focus on the graphs for which a weakly convex domination number equals a connected domination number. We also study the influence of the edge removing on the weakly convex domination number, in particular we prove that the weakly convex domination number is an interpolating function.\\
\end{abstract}


\section{Introduction}
 All graphs considered in this paper are finite, undirected, simple and connected. 
 Let   $G =(V ,E)$ be a connected graph of order $ \mid V \mid=n$, where $V=V(G)$ is the set of the vertices of $G$ and $E=E(G)$ denotes the set of edges of $G$.
For a vertex $v \in V$, the \emph{open neighbourhood} $N_{G}(v)$ is the set of all vertices adjacent to $v$ and the \emph{closed neighbourhood} $N_{G}[v]=N_G(v)\cup \{v\}$.  The \emph{degree} of a vertex $v$ is  $d_{G}(v)=|N_{G}(v)|.$ 
We say that a vertex $v$ is a {\em simplicial vertex} if $N_G[v]$ is a complete graph. 
A vertex $v$ is an \emph{end-vertex} (or a {\em leaf}) of $G$ if $v$ has exactly one neighbour in $G.$ The set of all end-vertices in $G$ is denoted by $V_L$ and $n_L=|V_L|$. 
A vertex $v$ is called a \emph{support} if it is adjacent to an end-vertex.  The set of all supports of $G$ is denoted by $V_S.$
{\it A cut-vertex}
in $G$ is a vertex $x\in V(G)$ such that the number of components of
 $G-\{x\}$ is bigger than the number of components of $G$. The set of all cut-vertices of $G$ is denoted by $V_C$.
Let $g(G)$ denotes the girth of $G$ that is, the length of the shortest cycle in $G$.

 A subset $D$ of $V$ is \emph{dominating} in $G$ if every vertex of $V-D$ has at least one neighbour in $D.$  The set  $D$ is \emph{connected dominating} in $G$ if  it is dominating and the subgraph $G[D]$ induced by $D$ is connected.  
The minimum cardinality of a connected dominating set of $G$ is a \emph{connected domination number} of $G$ and is denoted by $\gamma_c(G).$  

Dominating sets  have been intensively studied since the fifties and 
the main interest  is due to their relevance on both theoretical and practical side. 
Several variants of the classical concept of domination were obtained where additional conditions on the subgraph induced by the dominating set were added. Connected dominating sets,  introduced      in \cite{sampath},   are one of these variants, and have  useful applications   in the
wireless  networks   context (see \cite{kim}).   In addition,  the design of communication networks involves considering short distances between nodes in order to get quick transmition.

 The {\it distance} $d_{G}(u,v)$ between two
vertices $u$ and $v$ in a connected graph $G$ is the length of a
shortest $(u-v)$-path in $G.$ A $(u-v)$-path of length $d_{G}(u,v)$ is
called $(u-v)$-{\it geodesic}. 
The {\it diameter} of a graph  $G$, denoted as diam$(G)$}, is defined as the maximum of distance over all  pair of vertices.
A set $X$ is {\it weakly convex} in $G$ if for any two vertices
$a,b\in X$  there exists an $(a-b)$--geodesic such that all of its
vertices belong to $X$. 
A set $X\subseteq V$ is a {\it  weakly
convex dominating set}  if $X$ is weakly convex and dominating. 

The
{\it  weakly convex domination number of a graph $G$} 
denoted by
$\gamma_{\rm wcon}(G)$
equals to the minimum cardinality
of a weakly convex dominating set in $G.$
It
 was first introduced by Jerzy Topp  in 2002 and 
 formally defined and studied in
\cite{magda}. 
This concept improves the applications of  connected domination in the design of communication networks, by guarantee that the connections trough the nodes of the dominating set are the shortest.
 
  In this work we investigate the relationship between the weakly convex domination and the connected domination.
  Moreover, we     study  edge removing  and its effect on the weakly convex number  for some graphs. Related to this, we have the idea of interpolation.    {In  the 80s of the last century,  the study of the   interpolation properties started. It is often thought that the origin was the   problem whether a graph $G$  containing spanning trees having $k$ and $l$ 
end-vertices, respectively,   also must contain a spanning tree with $r$ end-vertices for
every integer $r$ such that $k < r < l$ . Several authors published some results of interpolation theorems on various kinds of graph parameters with respect to the set of all spanning trees
and some classes of spanning subgraphs of a given graph.} In this paper, we  conclude that for each connected graph $G,$ the image of the function $\gwc$ over the set of all spanning trees is    an  interval, what means  that $\gwc$ is an interpolating function.

The paper 
is organized as follows.
 In  Section 2, we present  relations  between the weakly convex domination and the connected domination. In  particular, 
 we look for conditions on the graph $G$ under which we get $\gwc(G)= \gc(G).$ First, we prove that there are graphs for which the difference between $\gwc(G)$ and $\gc(G)$ can be arbitrarily large.
 In the next section, we show some examples of families of graphs $G$ with  equality  $\gwc(G)=\gc(G)$, e.g.   distance-hereditary graphs (in particular block graphs) and  cacti graphs (in particular unicyclic graphs). 
  Moreover we  study the
  chordal graphs. This kind of graphs have an extensive literature and applications  (see for example \cite{chordal}, \cite{martinez} \cite{sch}). 
We complete Section 3 focussing our attention on induced subgraphs; more specifically, we give
  conditions for the  weakly convex domination number  to be equal to the connected domination number for a graph and every its induced subgraph.  
Section 4 is devoted to
 study the influence of the edge removing on the weakly convex domination number, in particular we show that a weakly convex domination number is an interpolating function.


\section{Connected and weakly convex domination numbers}
 
 This section we start with proving that there are graphs for which the difference between $\gwc$ and $\gc$ can be arbitrarily large.

  \begin{thm} For any $k\in \mathbb{N}$ and $k\geq 6$, there exists a graph $G$ such that $$\gwc(G)-\gc(G)=k.$$
\end{thm}
 
\begin{proof}
We begin with a cycle $C_{2k+6}=(x_1,v_1,v_2,\ldots ,v_{k+2},x_2,u_{k+2},u_{k+1},\ldots u_1, x_1).$ Next we add to this cycle: new vertices $x'_1,v'_1,v'_2,\ldots ,v'_{k+2},x'_2$ and edges $x_1x'_1,x_2x'_2$ and $v_iv'_i$ where $1\leq i\leq k+2$, and we also add edges $v_ju_j$ for $1\leq j\leq k+2$, and the edge $v_1u_3$. The final graph is illustrated in the Figure \ref{rys_2}.  Let $D$ be a minimum weakly convex dominating set of $G$. All supports of $G$ belong to $D$. Moreover, the distance between $v_1$ and $x_2$ equals to $k+1$ and $u_3,u_4,u_5,\ldots ,u_{k+2}$ belong to the shortest $(v_1-x_2)$-path, so they also belong to $D$. Hence  $\gwc (G)=2k+4$.

It is easy to observe that 
 $\{x_1,v_1,v_2,\ldots ,v_{k+2},x_2\}$ is the minimum connected dominating set of $G.$ Thus $\gc(G)=k+4$ and $\gwc(G)-\gc(G)=k.$ \end{proof}
\begin{figure}[h]
\begin{center} 
	\begin{tikzpicture}[scale=0.7]

\fill(0,1) circle(3pt);
\fill(10,1) circle(3pt);
\draw (0,1)--(0,2);
\draw (10,1)--(10,2);

\draw(1,2)--(3,0);

\filldraw[fill=white](0,2) circle(3pt);
\filldraw[fill=white](10,2) circle(3pt);
\draw (6,0)--(1,0)--(0,1)--(1,2)--(6,2);
\draw[dotted](6,2)--(8,2);
\draw (8,2)--(9,2)--(10,1)--(9,0)--(8,0);
\draw[dotted](6,0)--(8,0);

\draw(-0.3,2.3)node{$_{x'_1}$};
\draw(10.2,2.3)node{$_{x'_2}$};
				
\foreach \nn in {1,2,3,4,5,6,8,9}
			{	\filldraw[fill=black] (\nn,2) circle (3pt);\draw(\nn,2)--(\nn,3);\draw (\nn,0)--(\nn,2);
				\filldraw[fill=white] (\nn,3) circle (3pt);}
				\foreach \nn in {1,2,3,4,5,6}{
\draw(\nn-0.3,3.3)node{$_{v'_{\nn}}$};
				}\draw(8-0.55,3.3)node{$_{v'_{k+1}}$};\draw(9-0.5,3.3)node{$_{v'_{k+2}}$};
				\filldraw[fill=white] (1,0) circle (3pt);\filldraw[fill=white] (2,0) circle (3pt);
				\foreach \nn in {3,4,5,6,8,9}
				\filldraw[fill=black] (\nn,0) circle (3pt);
\draw(-0.2,0.6)node{$_{x_1}$};\draw(10.2,0.6)node{$_{x_2}$};
\draw(0.7,2.2)node{$_{v_1}$};\draw(1.7,2.2)node{$_{v_2}$};\draw(2.7,2.2)node{$_{v_3}$};\draw(3.7,2.2)node{$_{v_4}$};\draw(4.7,2.2)node{$_{v_5}$};\draw(5.7,2.2)node{$_{v_6}$};\draw(7.3,2.2)node{$_{v_{k+1}}$};\draw(8.5,2.2)node{$_{v_{k+2}}$};
\draw(1,-0.4)node{$_{u_1}$};\draw(2,-0.4)node{$_{u_2}$};\draw(3,-0.4)node{$_{u_3}$};\draw(4,-0.4)node{$_{u_4}$};\draw(8,-0.4)node{$_{u_{k+1}}$};\draw(9,-0.4)node{$_{u_{k+2}}$};\draw(5,-0.4)node{$_{u_5}$};\draw(6,-0.4)node{$_{u_6}$};
\end{tikzpicture}
\end{center}
\caption{Construction of a graph $G$ such that $\gwc(G)-\gamma_c(G)=k$. Black vertices denote the minimum weakly convex dominating set of $G$.}\label{rys_2}
\end{figure}
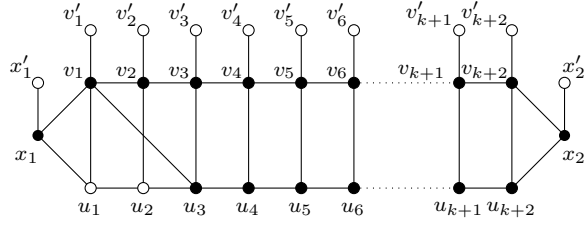
 
In \cite{sampath} the following result was proved for a connected domination number of any connected graph $G.$

\begin{thm} \label{sam} \cite{sampath}
For any connected graph $G$ with $n\geq 3$ vertices and $m$ edges is $\gamma_c(G)\leq 2m-n$ with equality if and only if $G$ is a path. 
\end{thm}

Similar result we can prove for weakly convex domination number.

 \begin{thm} 
For any connected graph $G$ with $n\geq 3$ vertices and $m$ edges is $\gwc(G)\leq 2m-n$ with equality if and only if $G$ is a path or a cycle $C_p$ with $p\geq 7.$ 
\end{thm}
\begin{proof}
Assume first that $G$ is a tree. Then $\gwc(G)=n-n_L$ $=2(n-1)-n+2-n_L$ $=2m-n+2-n_L \leq 2m-n$ with equality $\gwc(G)=2m-n$ when $n_L=2,$ what means that $G$ is a path.

Assume now $G$ is not a tree; thus $m\geq n.$ Let first $\gwc(G)\leq n-1.$ Then we have $\gwc(G)\leq n-1 = 2n-n-1<2m-n.$ 
Let now $\gwc(G)=n.$ Then we have $\gwc(G)=n=2n-n\leq 2m-n$ with equality for $m=n.$ The only case where $\gwc(G)=n$ and $m=n$ happens if $G=C_p, p~\geq~7.$
\end{proof}

%
%


%
%

%

We continue with showing that we can get   
 $\gc(G)=\gwc(G)$   
 if we consider graphs for which  the subgraph induced by a minimum connected dominating set $D$  has a small diameter. A connected perfect dominating set  $D$ is a connected  dominating set  where each vertex in $G$  is dominated by exactly one vertex  of $D.$

\begin{lem}\label{diameter} If $G$ is a connected  graph with a  minimum connected dominating set $D$  such that either 
${\rm diam}(G[D])\leq 2$ or  $D$ is a perfect   connected dominating set with ${\rm diam}(G[D])=3,$ then $\gc (G)=\gwc (G)$.
\end{lem}
\begin{proof}   
Suppose  there exist $x,y\in D$ such that $d_{G[D]}(x,y)>d_G(x,y)$. Of course $d_G(x,y)>1;$ otherwise $1=d_G(x,y)=d_{G[D]}(x,y).$
If  ${\rm diam}(G[D])\leq 2$,  also $d_{G}(x,y)\leq 2$  for  all pair of different vertices $x,y\in D,$  a contradiction. 
Now, assume that $D$ is perfect dominating set with ${\rm diam}(G[D])=3.$  Then  $d_{G[D]}(x,y)\leq 3$  for  all pair of different vertices $x,y\in D.$
If $d_G(x,y)=2<d_{G[D]}(x,y)=3,$ then there is a vertex $z\notin D$ such that $z\in N(x)\cap N(y),$  a contradiction with the fact that $D$ is perfect dominating set. 
\end{proof}





Note that there are no graphs $G$ for which $\gamma_c(G)=n$ or $\gamma_c(G)=n-1.$
In fact, in \cite{sampath} the following result was proved.

\begin{thm} \cite{sampath}
For any connected graph $G$ with at least three vertices is $\gamma_c(G)\leq n-2$ with equality  $\gamma_c(G)= n-2$
if and only if $G = P_n$ or $G = C_n$.
\end{thm}

However, there are some graphs $G$ for which $\gamma_{wcon}(G)=n.$ Some of these graphs were characterized in \cite{lemanska} in the following theorem.

\begin{thm}\label{twr2} \cite{lemanska} If $G$ is a connected graph with no end-vertex such that $g(G)\geq 7,$ then $\gamma_{wcon}(G)=n.$
\end{thm}

 By definition we have that in general, $\gwc(G)\geq \gc(G)$ for any graph $G.$  As  consequence of  Theorems   above
we get necessary condition for $\gamma_{wcon}(G)= \gamma_{c}(G).$  
\begin{wn} \label{necessary}   If $\gamma_{wcon}(G)= \gamma_{c}(G),$ 
 then either $G$ has end-vertices  or $g(G)<~7.$
\end{wn}

$ $

We have a more general result for   connected graph with  $g(G)\geq 7.$  
We begin with  the following  observation.

\begin{obs}\label{ob1}
Let $G\neq K_n$ be a graph of order $n\geq 3$. If $D$ is a minimum connected or weakly convex
dominating set of $G,$ then every cut-vertex belongs to $D$ and  no simplicial vertex belongs to $D.$ 
\end{obs}

\begin{wn}\label{cut_simplicial} 
If $G$ contains only cut-vertices and simplicial vertices,   then $V_C$ is a minimum connected (and also weakly convex) dominating set of $G$ and $\gamma_c(G)=\gamma_{\rm wcon}(G)$. 
\end{wn}

Now we prove the following result.

\begin{thm}\label{girth}Let $G$ be  a connected graph with $g(G)\geq 7.$ Then,

\begin{enumerate}
 \item $\gwc (G)=n-n_L$; 
 \item $\gwc (G)= \gamma_{c}(G)$ 
if and only if for every $u\in V,$ $u$ is either an end-vertex or a cut-vertex.
 \end{enumerate}
\end{thm} 
\begin{proof} (1). Every connected graph $G$ with $g(G)\geq 7$ can be obtained from some graph $G'$, which fulfills conditions from the Theorem~\ref{twr2}, by adding only cut and end-vertices. As $\gwc(G') = |V(G')|$ and every end-vertex is simplicial, we conclude from the Observation~\ref{ob1} that  $\gwc(G) = n - n_L$.

(2). 
Assume $\gwc (G)= \gc (G).$ From (1) we have $\gwc (G)=\gc (G)  =n - n_L.$
 Suppose there is a vertex $u$ such that $u$ is neither an end-vertex nor cut-vertex. Then $d_G(u)\geq 2$ and for any $x,y\in N_G(u)$ there exists $(x-y)$-path not containing $u.$ Then $V-(V_L\cup \{u\})$ is a connected dominating set of $G$   and $\gc (G)\leq  |V-(V_L\cup \{u\})|=n-(n_L+1)<n-n_L,$ a contradiction.
 
Now assume  for every $u\in V,$ $u$ is either an end-vertex or a cut-vertex.  From Observation \ref{ob1}, the result holds. 
\end{proof}





\section{Some graphs $G$ with equality  $\gwc(G)=\gc(G)$ }

 In this section we provide conditions under which we have equality  $\gwc(G)=\gc(G)$ in particular families of graphs $G$.
We begin with cacti.
 \subsection{Cacti} 
 A \emph{cactus} is a connected graph in which any two simple cycles have at most one common vertex. Equivalently, it is a connected graph in which every edge belongs to at most one simple cycle.
Notice that a unicyclic graph is a cactus with only one cycle and a trees is a cactus having no cycle.
 
\begin{thm}\label{thm:cactus} Let $G$ be a cactus. Then $\gc(G) = \gwc (G)$ if and only if:
\begin{enumerate}
\item for every  cycle in $G,$ $C = C_i,\ i \in \{5,6\}$  we have $d(v) \geq 3$ for every $v \in V(C_i)$ or $C_i$ has two adjacent vertices of degree $2$ and
\item for every $C = C_i,\ i \geq 7$ in $G$, $d(v) \geq 3$ for every $v \in V(C)$.
\end{enumerate}
\end{thm}

\begin{proof}
Assume (1) and (2) hold. Let $D$ be a minimum connected dominating set of $G$. We show that $D$ is also weakly convex. Suppose it is not true. Then there exist vertices $x$ and $y$ such that $\{x,y\}\subset D$ and $d_{G[D]}(x,y)> d_{G}(x,y)$. This implies that $x$ and $y$ belong to a cycle $C$ of the length at least 5. 
Since (2) holds, $C$ can have the length 5 or 6.
 In fact, if $C$ is longer than $6$, then from $(2)$ every vertex on $C$ is a cut-vertex, hence from Observation \ref{ob1}, $V(C)\subseteq V_C \subseteq D$, which contradicts $d_{G[D]}(x,y)> d_{G}(x,y)$.
  Hence, only one vertex from $C$ does not belong to $D$, let us say $z\not \in D$. Notice that $d_G(z)=2$ and $z$ is a common neighbour of $x$ and $y$. The vertices from $V(C)-\{z\}$ form an $(x-y)$-path $P$ such that $V(P)\subseteq D$. Since $D$ is minimum, at most one vertex of $P$ has degree 2 (if not we could find a smaller connected dominating set of $G$). It gives a contradiction with (1) and finally $\gc (G)=\gwc (G)$.

Conversely, let $D$ be a minimum weakly convex dominating set of $G$. Since $\gc (G)=\gwc (G)$, $D$ is also a minimum connected dominating set of $G$. Suppose first that (1) does not hold. Thus there exists a cycle $C_i$, $i\in \{5,6\}$, such that $C_i$ has at least one vertex of degree 2 and the set of the vertices of degree 2 of $C_i$ is independent. Since vertices of degree 3 of $C_i$ (as cut-vertices, by Observation \ref{ob1}) belong to $D$ and since $D$ is weakly convex and $i\in \{5,6\}$, we obtain that $V(C_i)\subseteq D$. 
  We get   a contradiction becase $D-\{v\}$, where $d_G(v)=2$ and $v\in V(C_i)$,  is a connected dominating set of $G$. Next, suppose that (2) does not hold. Thus there exists a vertex $v$ of degree 2 on a cycle $C_i$, $i\geq 7$. Since $D$ is weakly convex, $V(C_i)\subset D$ and then $D-\{v\}$ is a connected dominating set of $G$, a contradiction.
\end{proof}

 
%

\subsection{Distance-hereditary  graphs}   
A {\em distance-hereditary graph} is a connected graph in which every connected induced subgraph is isometric (that is, the distance of any two vertices in any connected induced subgraph equals their distance in the graph). From the definition we get the following
\begin{lem} If $G$ is a distance-hereditary graph, then $\gwc(G)=\gc(G).$ 
\end{lem}
\begin{proof} Let  $D$ be a connected dominating set of $G.$ If $|D|=1,$ then  $\gwc(G)=\gc(G).$ 
Now let $x,y $ be two different vertices in $D.$ 
Then there is an induced $(x-y)$-path with all its vertices in $D.$ Using the fact that  $G$ is distance-hereditary we conclude that  $d_{G[D]}(x,y)=d_G(x,y).$ 
Then $D$ is  weakly convex and we can get $\gwc(G)=\gc(G).$ 
\end{proof}

Notice that the converse is not true, i.e. there exist graphs $G$ with equality $\gwc(G)=\gc(G)$ which are not distance-hereditary; the example of such a graph can be a corona of a cycle $C_7$, i.e. $G=C_7 \circ K_1$.
%

A block in a graph  $G$ is a maximal
connected subgraph  $H$ of $G$ such that $H$ does not contain any
cut-vertex of $H$. A 
{\it block graph} is a connected graph $G$ such that
every block in $G$ is a complete graph). 
Since a block graph is a distance-hereditary graph, we conclude
the following

\begin{wn}
If $G$ is a connected block graph, then 
$\gc(G) =\gwc(G)$.
\end{wn}

%
%
 
We say that a graph $G$ is $H$--free if $G$ does not contain $H$ as an induced subgraph.
In particular, $P_4$--free graph is called a {\it cograph}. Since a cograph is a distance-hereditary graph, we conclude
the following

\begin{wn}
If $G$ is a connected cograph, then 
$\gc(G) =\gwc(G)$.
\end{wn}


\subsection{Chordal graphs} Now we foccus our attention on chordal graphs. Recall that  a graph is chordal  if every cycle of length at least 4 has a chord.


\begin{figure}[!htb]
\begin{center} 
	\begin{tikzpicture}[scale=0.7]

 \filldraw[fill=black](3,2) circle(3pt);

 \fill(3,1) circle(3pt);
 \draw(3,2)--(3,1);
 \filldraw[fill=black](1.5,3) circle(3pt);

\fill(1.5,2) circle(3pt);
\draw(1.5,2)--(1.5,3);

 \filldraw[fill=black](-0.5,2) circle(3pt);
 
\fill(0,1) circle(3pt);
\draw(-0.5,2)--(0 ,1);

 \filldraw[fill=black](0,-0.5) circle(3pt);

\fill(0.8,-0.2)  circle(3pt);
\draw(0.1,-0.5)--(0.8,-0.2);

\draw (0,1)--(1.5,2)--(3,1)--(2.2,-0.2)--(0.8,-0.2)--cycle;

\draw(0,1)--(2.2,-0.2);

\draw(1.5,2)--(2.2,-0.2);
\filldraw[fill=black](0.8,-0.2) circle(3pt);
 \filldraw[fill=black](2.2,-0.2) circle(3pt);
\end{tikzpicture}
\end{center}
\caption{Graph $H^*$.}
\label{rr}
\end{figure}

\begin{thm} \label{chordal} Let $G =(V ,E)$ be a connected chordal graph of order $n$. If  $G$ is  $H^*-$free (where $H^*$ is a graph from Figure~\ref{rr}), then  $\gc (G)=\gwc (G).$  
\end{thm}
 \begin{proof} 
  Suppose that the statement is not true and let $G$ be a minimal counterexample, i.e., $G$ is a connected chordal graph without induced $H^*,$ but  $\gc (G)<\gwc (G)$.
Let $D$ be any minimum connected dominating set of $G$ (obviously $D$ is not weakly convex). 
  
We consider cases depending on ${\rm diam}(G[D])=d.$  If $d\leq 2$, then $\gc(G)=\gwc(G)$, a contradiction.  If $d=3,$ then there exist $x,y\in D$ such that $d_{G[D]}(x,y)>d_G(x,y).$ 
  Observe that  $d_{G[D]}(x,y)=3=d$ otherwise $d_G(x,y)=1$, a contradiction.
  \begin{figure}[!htb]
\begin{center} 
	\begin{tikzpicture}[scale=0.6]

\fill(3,1) circle(3pt);
\fill(1.5,2) circle(3pt);
\fill(0,1) circle(3pt);

\draw (0,1)--(1.5,2)--(3,1)--(2.2,-0.2)--(0.8,-0.2)--cycle;
\draw(0,1)--(2.1,-0.1);
\draw(1.5,2)--(2.1,-0.1);
\filldraw[fill=black](0.8,-0.2) circle(3pt);
\filldraw[fill=white](2.2,-0.2) circle(3pt);

\draw(2.5,-0.3)node{$a$};\draw(3.3,0.7)node{$y$};\draw(0.4,-0.3)node{$x$};
\end{tikzpicture}
\end{center}
\caption{Graph $H'_a$.}
\label{rrr}
\end{figure}
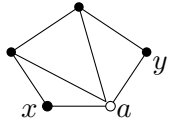

 Thus $d_G(x,y)=2$  and there is $a\notin D$ such that $a\in N_G(x)\cap N_G(y).$
 In this case we find a  cycle $C_5$  in $G$ such that four vertices belong to $D$ and the other, namely $a$, does not belong to $D.$

Using the fact that $G$ is chordal and $d=3$ we get that the  induced subgraph $G[C_5]$  contains a subgraph   isomorphic to  $  H_a'$ shown in the Figure~\ref{rrr} where $a\notin D$ and the rest black nodes belong to $D.$
  Since $D$ is not weakly convex, for every  vertex $v \neq a$ of $C_5$  exists $v'\in N_G(v)$ such that $v'\notin N_G(V(C_5)-\{v\}),$   (otherwise $(D-\{v\})\cup \{a\}$ would be a weakly convex dominating set of $G$ of cardinality $\gc(G)$).

 From above and using that $G$ is chordal and  $d_{G[D]}(x,y)= 3,$     we find an induced subgraph isomorphic to  {$H^*$},  a contradiction.
 
 Now we assume  $d>3$.

There are $x,y \in D$ such that $d_{G[D]}(x,y)=d> d_G(x,y).$ Because $t=d_G(x,y)>1$ there exist a $(x-y)$-geodesic in $G[D]$, say $$P=(x=z_0, z_1,z_2, \cdots ,z_{d-1},z_d=y)$$  and another  $(x-y)$-geodesic in $G$, namely  $P'= (x,w_1, w_2, \cdots  ,w_{t-1},y)$ such  that  $w_i\not \in D$  for $i\in \{1, \cdots, t-1\}$. 
 Then we have   a cycle $ C_{d+t}$, where $d+t\geq 6$.   

  Using again   the chordal condition and the fact that both $P$ and $P'$ are geodesic, we get that the possible chords are edges of the form $z w$ with  $w\in V(P')$  and $z\in V(P).$ Note that every two consecutive vertices $z_i, z_{i+1}$ of $P$ must share at least  one neighbour in $P'.$
 
{\em Case 1. }If there exists a  vertex $a$ of $P'-\{ x,y\}$ such that   
four consecutive vertices    can be found in $N[a]\cap P,$  then $G$ has  a subgraph isomorphic to    $H_a'$ and, with similar arguments as before we conclude that $G$ has an induced subgraph isomorphic to {$H^*$}, a contradiction.

{\em Case 2.} Otherwise, since $G$ is chordal, there are two consecutive vertices $w_i, w_{i+1}$ and a sequence of five vertices in $P$ namely $z_j, \cdots, z_{j+4}$   such that  $  N_G[w_i]\cap P   = \{z_j, z_{j+1}, z_{j+2}\}$ and  $  N_G[w_{i+1}]\cap P   = \{z_{j+2}, z_{j+3}, z_{j+4}\}.$  

Denote by $H'=G[\{   w_i, \, w_{i+1}, \, z_j, z_{j+1}, z_{j+2}, z_{j+3}, z_{j+4}\}]$ the  induced subgraph of $G.$

Note that with an argument similar to the  used above we get that:  If the edge ${w_iz_{j+4}}$ is in $H'$ then also  the edge ${w_iz_{j+3}}$ is in $H'$  and then we have the Case 1.
  The same occurs if the edge ${w_{i+1}z_{j}}$ is in $H'.$

  \begin{figure}[!htb]
\begin{center} 
\begin{tikzpicture}[scale=0.5]

 \fill(3,1.8) circle(3pt);

 \draw(3.6,1.8)node{$z_{j+3}$};
\fill(1.5,2) circle(3pt);

  \draw(1,2.3)node{$z_{j+2}$};

\fill(0,1) circle(3pt);

  \draw(-0.6,1)node{$z_{j+1}$};

\fill(-0.2,-0.2)  circle(3pt);

  \draw(-0.3,-0.5)node{$z_{j}$};

\fill(4.3,0)  circle(3pt);
  \draw(4.9,0.1)node{$z_{j+4}$};

 \filldraw[fill=white](3.5,-1) circle(3pt);
  \draw(3.6,-1.3)node{$w_{i+1}$};

\draw (-0.2,-0.2)--(0,1)--(1.5,2)--(3,1.8)--  (4.3,0)--(3.5,-1) --(1.3,-1.2)--cycle;

 \draw(0,1)--(1.3,-1.2);
\draw(1.5,2)--(3.5,-1);
\draw(1.5,2)--(1.3,-1.2);
\draw(3,1.8)--(3.5,-1);
 
 \filldraw[fill=white](1.3,-1.2) circle(3pt);
\filldraw[fill=white](3.5,-1) circle(3pt);
 \draw(1.8,-1.4)node{$w_i$};

\end{tikzpicture}
\end{center}
\caption{Graph $H'$.}
\label{rrrr}
\end{figure}

Using similar arguments as before we get that $H'$ is an induced subgraph of $G$  and also $H'_{w_{i},w_{i+1}}.$

 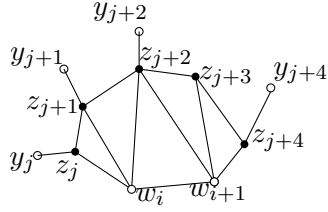
\begin{figure}[!htb]
\begin{center} 
\begin{tikzpicture}[scale=0.5]

  \fill(3,1.8) circle(3pt);

 \draw(3.8,1.8)node{$z_{j+3}$};

 \filldraw[fill=white](1.5,3) circle(3pt);
\fill(1.5,2) circle(3pt);
 \draw(1.5,2)--(1.5,3);
 \draw(2.2,2.3)node{$z_{j+2}$};
 \draw(1,3.5)node{$y_{j+2}$};
  
 \filldraw[fill=white](-0.5,2) circle(3pt);
\fill(0,1) circle(3pt);
   \draw(-0.5,2)--(0 ,1);
  \draw(-0.8,1)node{$z_{j+1}$};
 \draw(-1.2,2.3)node{$y_{j+1}$};

  \filldraw[fill=white](-1.2,-0.3) circle(3pt);
 \fill(-0.2,-0.2)  circle(3pt);
 \draw(-1.2,-0.3)--(-0.2,-0.2);
 \draw(-0.4,-0.7)node{$z_{j}$};
 \draw(-1.6,-0.5)node{$y_{j}$};

 \filldraw[fill=white](5 ,1.5) circle(3pt);
\fill(4.3,0)  circle(3pt);
  \draw(5.2,0.1)node{$z_{j+4}$};
 \draw(4.3,0)--(5,1.4);
 \draw(5.8,1.9)node{$y_{j+4}$};
  
 \filldraw[fill=white](3.5,-1) circle(3pt);
 \draw(3.6,-1.3)node{$w_{i+1}$};

\draw (-0.2,-0.2)--(0,1)--(1.5,2)--(3,1.8)--  (4.3,0)--(3.5,-1) --(1.3,-1.2)--cycle;

 \draw(0,1)--(1.3,-1.2);
\draw(1.5,2)--(3.5,-1);
\draw(1.5,2)--(1.3,-1.2);
\draw(3,1.8)--(3.5,-1);

 \filldraw[fill=white](1.3,-1.2) circle(3pt);
\filldraw[fill=white](3.5,-1) circle(3pt);
 \draw(1.8,-1.4)node{$w_i$};
\end{tikzpicture}
\end{center}
\caption{Graph $H'_{w_{i},w_{i+1}}.$}
\label{rrrr}
\end{figure}

If we delete $y_{j+4}$ and $z_{j+3}$ from $H'_{w_{i},w_{i+1}},$ then we have  an induced subgraph isomorphic to $H^*$.
From the fact that $G$ is $H^*-$free we conclude the result.

\end{proof}
 \subsection{Perfect graphs}

 We say that a graph $G$ is $(\gamma_c-\gamma_{\rm wcon})$-{\em perfect}  if 
 $\gamma_c(H)=\gamma_{\rm wcon}(H)$ for every connected induced subgraph $H$ of $G.$

The following result shows that chordal $H^* -$free graphs (where $H^*$ is a graph from Figure~\ref{rr}) are  $(\gamma_c-\gamma_{\rm wcon})$-perfect.
 

\begin{thm} \label{Hchordal} Let $G =(V ,E)$ be a connected chordal graph of order $n$. Then $G$  is $(\gamma_c-\gamma_{\rm wcon})$-perfect  if and only if     $G$ is an $H^* -$free  graph.
\end{thm}
\begin{proof} 
 Suppose   $\gc (H)=\gwc (H)$ for every induced subgraph $H$ of $G.$ Then  $H\neq H^*$ since $\gc(H^*)<\gwc(H^*).$  
  
 Now, suppose $G$ is $H^* -$free. Let $H$ be an induced subgraph of $G$. Hence, $H$ is chordal and $H^*-$free. By Theorem~\ref{chordal}, $\gc(H)=\gwc(H)$ and thus, $G$ is $(\gamma_c-\gamma_{\rm wcon})$-perfect. 
 \end{proof}

$ $

Now we generalize the necessary condition for any $(\gamma_c-\gamma_{\rm wcon})$-perfect graph.
 
 For a cycle $C$ of $G,$ denote by $H$ a subgraph $G[N[V(C)]]$ induced by the closed neighbourhood of  the vertices from the cycle $C.$
  In general, we have the following result. 
 \begin{lem}\label{lemma:perfect}
 If a graph $G$ is a perfect graph, then $G$ does not contain induced cycle greater than six and for every cycle (not necessarily induced) $C$ of length five or six one of the two following conditions hold:
 \begin{enumerate}
 \item there are two consecutive vertices of $C$ such that they are not cut-vertices of~$H$;
 \item for every vertex $v$ of $C$ such that $v$ is a cut-vertex of $H$, their neighbours on $C$ are connected by an edge or have a common neighbour on $C$ different then $v$. 
 \end{enumerate} 
 \end{lem}

 \begin{proof}
 Assume $G$ to be a perfect graph. 
 If there is an induced cycle $C_p$ in $G$ such that $p\geq 7,$ then, since $\gamma_c(C_p)<\gamma_{\rm wcon}(C_p),$ $G$ is not perfect; so we conclude that for a perfect graph $G$ there is no induced cycle of length greater than six. 
 
  Denote by $A$ the set of vertices of $C$ such that they are cut-vertices of $H$ and let $B=V(C)-A$.
 
 Assume $C$ is a  cycle of length five in $G.$ Suppose that (1) does not hold; then $|A|\geq 3,$ $|B|\leq 2$  and $B$ is an independent set. Denote $C=(a,b,c,d,e,a);$ without loss of generality let $a,b,d \in A.$ Since $\{a,b,d\}\subseteq A$, they have their private neighbours with respect to $V(C),$ $a',b',d',$ respectively and $\{a'b', b'd', a'd'\}\notin E(G).$ Suppose that also (2) does not hold - then $C$ is an induced cycle of $G$. In this case $F = G[\{a,a',b,b',c, d,d', e\}]$ is an induced subgraph of $G$ such that $\gamma_c(F)<\gamma_{\rm wcon}(F),$ a contradiction.
 
 Now let $C$ be a cycle of length six in $G.$ If (1) does not hold, then  $|A|\geq 3,$ $|B|\leq 3$  and $B$ is an independent set. Denote $C=(a,b,c,d,e,f,a);$ without loss of generality let $\{a,c,e\} \subseteq A.$  Similarly like in the previous case, these vertices have their private neighbours with respect to $V(C)$, $a', c', e',$ respectively such that $\{a'c', c'e', a'e'\}\notin E(G).$ If also (2) does not hold, then $C$ is an induced cycle of $G$ or there is one chord between vertices of $A$ in $C$. In both cases $F' = G[\{a,a',b,c,c', d,e,e', f\}]$ is an induced subgraph of $G$ such that $\gamma_c(F')<\gamma_{\rm wcon}(F'),$ a contradiction.
 \end{proof}
 
 $ $

Notice that there are graphs which satisfy conditions (1) and (2), but are not
perfect. Example of such a graph is graph $G$ from Figure \ref{example_not_perfect} for which there exists a
cycle $C$ of length five such that there are two consecutive non-cut-vertices $a,b$ of $G$ in
$C$; but $G$ is not perfect (observe that the set of all support vertices $V_S$ is a minimum connected dominating set of $G$ and $V_S\cup \{a,b\}$ is a minimum weakly convex set of $G$).

 \begin{figure}[!htb]\label{rysx}

\begin{center} 
	\begin{tikzpicture}[scale=0.6]
\draw(0,1)--(1,1)--(2,0)--(1,-1)--(0,-1);
\draw(0,-1)--(-1,0)--(0,1);
\draw(0,1)--(-1,1)--(-1,0);
\draw(-1,0)--(-2,0);
\draw(-1,1)--(-2,1);
\draw(1,1)--(1,-1);
\draw(0,-1)--(0,-2);
\draw(1,-1)--(1,-2);
\draw(2,0)--(3,0);

\filldraw[fill=black](1,1) circle(3pt);
\filldraw[fill=black](2,0) circle(3pt);
\filldraw[fill=black](0,1) circle(3pt);
\filldraw[fill=black](1,-1) circle(3pt);
\filldraw[fill=black](-1,0) circle(3pt);
\filldraw[fill=black](0,-1) circle(3pt);
\filldraw[fill=black](-1,1) circle(3pt);
\filldraw[fill=black](-2,1) circle(3pt);
\filldraw[fill=black](-2,0) circle(3pt);
\filldraw[fill=black](0,-2) circle(3pt);
\filldraw[fill=black](1,-2) circle(3pt);
\filldraw[fill=black](3,0) circle(3pt);

\draw(0,1.3)node{$_a$};
\draw(1,1.3) node{$_b$};

\end{tikzpicture}
\end{center}
\caption{Example of a graph, which satisfies two conditions from Lemma \ref{lemma:perfect}, but is not perfect.}
\label{example_not_perfect}
\end{figure}
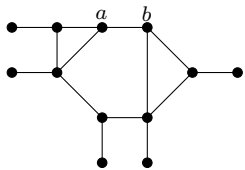


  
 \section{Edge removing}
 In this section we will analyze the influence of the edge removing on the   domination numbers for   graphs such   that every vertex is simplicial or cut-vertex
  and also for unicycle graphs. As consequence we will conclude that  weakly convex domination number is   an interpolating function.

First, we see that in general, deleting an edge can arbitrarily increase
and arbitrarily decrease the weakly convex domination number.

We say that edge $e\in E(G)$ is a {\em cut edge} if $G-e$ is not connected.

\begin{thm} For every integer $k$ there is a connected graph $G$ and an edge not cut $e\in E(G)$ such that $$\gwc (G-e)-\gwc (G)=k.$$
\end{thm}

\begin{proof}
If $k=0$, then $G=C_3$. For $k>0$ we show the construction of a graph $G$ such that $ \gwc (G-e)-\gwc (G)=k$ for some edge $e \in E(G)$. We begin with a cycle $C_{2k+2}=(x_1,v_1,v_2,\ldots ,v_k,x_2,u_k,u_{k-1},\ldots u_1,x_1)$.  Next we add to this cycle: new vertices $x'_1,x'_2,v'_i$ and edges $x_1x'_1,x_2x'_2, v_iv'_i, v_iu_i$ where $1\leq i\leq k$.

The final graph is illustrated in the Figure~\ref{rys1}. It is easy to observe that all supports of $G$ form a minimum weakly convex dominating set of $G$, so $\gwc (G)=k+2$. Let $G'=G-x_1v_1$. Hence, all supports again belong to the minimum weakly convex dominating set of $G'$. But the distance between $x_1$ and $x_2$ equals to $k+1$ and the only shortest $(x_1-x_2)$-path consists of vertices $u_1,u_2,\ldots ,u_k$. It implies that $\gwc (G)\geq 2k+2$. On the other hand, all but leaves vertices of $G$ form a weakly convex dominating set of $G'$, i.e. $\gwc (G)\leq 2k+2$. Hence, $\gwc (G')=2k+2$. Summing up, we obtained that $\gwc (G')-\gwc(G)=k$. 

Next, for $k<0$ we show the construction of a graph such that $ \gwc (G-e)-\gwc (G)=k$ for some edge $e$ of $G$. We begin with a cycle $C_{2|k|+6}=(x_1,v_1,v_2,\ldots ,v_{|k|+2},x_2,u_{|k|+2},u_{|k|+1},\ldots u_1,x_1).$  Next we add to this cycle: new vertices $x'_1,x'_2,v'_i$ and edges $x_1x'_1,x_2x'_2, v_iv'_i, v_iu_i$ where $1\leq i\leq |k|+2$, and edge $v_1u_3$.
The final graph is illustrated in the Figure~\ref{rys2}. Let $D$ be a minimum weakly convex dominating set of $G$. All supports of $G$ belong to $D$. Moreover, the distance between $v_1$ and $x_2$ equals to $|k|+1$ and $u_3,u_4,\ldots ,u_{|k|+2}$ belong to the shortest $(v_1-x_2)$-path, so they also belongs to $D$.  On the other hand, $V-(V_L\cup \{u_1,u_2\})$ is a weakly convex dominating set of $G$. Hence, $\gwc (G)=2|k|+4$. Let us consider a graph $G'=G-u_{|k|+2}x_2$
 It is easy to check, that the set of all supports of $G'$ is a minimum weakly convex dominating set of $G'$, and so $\gwc(G')=|k|+4$. Finally, $\gwc(G')-\gwc(G)=k$. 
\end{proof}

\begin{figure}[!htb]
\begin{center} 
	\begin{tikzpicture}[scale=0.6]

\fill(0,1) circle(3pt);
\fill(8,1) circle(3pt);
\draw (0,1)--(0,2);
\draw (8,1)--(8,2);

\filldraw[fill=white](0,2) circle(3pt);
\filldraw[fill=white](8,2) circle(3pt);
\draw (4,0)--(1,0)--(0,1)--(1,2)--(4,2);
\draw[dotted](4,2)--(6,2);
\draw (6,2)--(7,2)--(8,1)--(7,0)--(6,0);
\draw[dotted](6,0)--(4,0);

\foreach \nn in {1,2,3,4,6,7}
			{	\filldraw[fill=black] (\nn,2) circle (3pt);\draw(\nn,2)--(\nn,3);\draw (\nn,0)--(\nn,2);
				\filldraw[fill=white] (\nn,3) circle (3pt);;
				}
				\foreach \nn in {1,2,3,4,6,7}
				\filldraw[fill=white] (\nn,0) circle (3pt);
\draw(0,0.6)node{$_{x_1}$};\draw(8,0.6)node{$_{x_2}$};
\draw(0.7,2.2)node{$_{v_1}$};\draw(1.7,2.2)node{$_{v_2}$};\draw(2.7,2.2)node{$_{v_3}$};\draw(3.7,2.2)node{$_{v_4}$};\draw(5.5,2.2)node{$_{v_{k-1}}$};\draw(6.7,2.2)node{$_{v_k}$};
\draw(0,2.4)node{$_{x'_1}$};\draw(8,2.4)node{$_{x'_2}$};
\draw(0.7,3.4)node{$_{v'_1}$};\draw(1.7,3.4)node{$_{v'_2}$};\draw(2.7,3.4)node{$_{v'_3}$};\draw(3.7,3.4)node{$_{v'_4}$};\draw(5.5,3.4)node{$_{v'_{k-1}}$};\draw(6.7,3.4)node{$_{v'_k}$};
\draw(1,-0.4)node{$_{u_1}$};\draw(2,-0.4)node{$_{u_2}$};\draw(3,-0.4)node{$_{u_3}$};\draw(4,-0.4)node{$_{u_4}$};\draw(6,-0.4)node{$_{u_{k-1}}$};\draw(7,-0.4)node{$_{u_k}$};
\draw(4,-1.5)node{ $_G$};
\end{tikzpicture}\hspace{1cm}
\begin{tikzpicture}[scale=0.6]

\fill(0,1) circle(3pt);
\fill(8,1) circle(3pt);
\draw (0,1)--(0,2);
\draw (8,1)--(8,2);

\filldraw[fill=white](0,2) circle(3pt);
\filldraw[fill=white](8,2) circle(3pt);
\draw (4,0)--(1,0)--(0,1);\draw(1,2)--(4,2);
\draw[dotted](4,2)--(6,2);
\draw (6,2)--(7,2)--(8,1)--(7,0)--(6,0);
\draw[dotted](6,0)--(4,0);

\foreach \nn in {1,2,3,4,6,7}
			{	\filldraw[fill=black] (\nn,2) circle (3pt);\draw(\nn,2)--(\nn,3);\draw (\nn,0)--(\nn,2);
				\filldraw[fill=white] (\nn,3) circle (3pt);;
				}
				\foreach \nn in {1,2,3,4,6,7}
				\filldraw[fill=black] (\nn,0) circle (3pt);
\draw(0,0.6)node{$_{x_1}$};\draw(8,0.6)node{$_{x_2}$};
\draw(0.7,2.2)node{$_{v_1}$};\draw(1.7,2.2)node{$_{v_2}$};\draw(2.7,2.2)node{$_{v_3}$};\draw(3.7,2.2)node{$_{v_4}$};\draw(5.5,2.2)node{$_{v_{k-1}}$};\draw(6.7,2.2)node{$_{v_k}$};
\draw(0,2.4)node{$_{x'_1}$};\draw(8,2.4)node{$_{x'_2}$};
\draw(0.7,3.4)node{$_{v'_1}$};\draw(1.7,3.4)node{$_{v'_2}$};\draw(2.7,3.4)node{$_{v'_3}$};\draw(3.7,3.4)node{$_{v'_4}$};\draw(5.5,3.4)node{$_{v'_{k-1}}$};\draw(6.7,3.4)node{$_{v'_k}$};
\draw(1,-0.4)node{$_{u_1}$};\draw(2,-0.4)node{$_{u_2}$};\draw(3,-0.4)node{$_{u_3}$};\draw(4,-0.4)node{$_{u_4}$};\draw(6,-0.4)node{$_{u_{k-1}}$};\draw(7,-0.4)node{$_{u_k}$};
\draw(4,-1.5)node{ $_{G'=G-x_1v_1}$};
\end{tikzpicture}
\end{center}
\caption{Construction of a graph $G$ such that $ \gwc (G-e)-\gwc (G)=k$ for some edge $e$, where $k$ is a positive integer.}\label{rys1}
\end{figure}
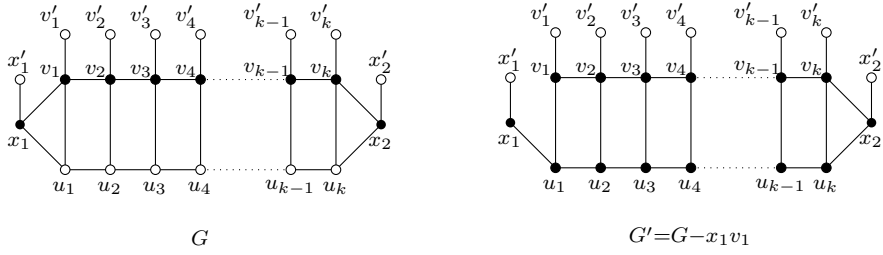
\begin{figure}[!htb]
\begin{center} 
	\begin{tikzpicture}[scale=0.6]

\fill(0,1) circle(3pt);
\fill(8,1) circle(3pt);
\draw (0,1)--(0,2);
\draw (8,1)--(8,2);

\draw(1,2)--(3,0);

\filldraw[fill=white](0,2) circle(3pt);
\filldraw[fill=white](8,2) circle(3pt);
\draw (4,0)--(1,0)--(0,1)--(1,2)--(4,2);
\draw[dotted](4,2)--(6,2);
\draw (6,2)--(7,2)--(8,1)--(7,0)--(6,0);
\draw[dotted](4,0)--(6,0);

\draw(0,2.4)node{$_{x'_1}$};\draw(8,2.4)node{$_{x'_2}$};
				\draw(5.7,3.4)node{$_{v'_{|k|+1}}$};\draw(7,3.4)node{$_{v'_{|k|+2}}$};
\foreach \nn in {1,2,3,4,6,7}
			{	\filldraw[fill=black] (\nn,2) circle (3pt);\draw(\nn,2)--(\nn,3);\draw (\nn,0)--(\nn,2);
				\filldraw[fill=white] (\nn,3) circle (3pt); 
				}
				\foreach \nn in {1,2,3,4}{
				\draw(\nn-0.3,3.4)node{$_{v'_{\nn}}$};}
				\foreach \nn in {3,4,6,7}
				\filldraw[fill=black] (\nn,0) circle (3pt);\filldraw[fill=white](1,0) circle(3pt);\filldraw[fill=white](2,0) circle(3pt);
\draw(0,0.6)node{$_{x_1}$};\draw(8,0.6)node{$_{x_2}$};
\draw(0.7,2.2)node{$_{v_1}$};\draw(1.7,2.2)node{$_{v_2}$};\draw(2.7,2.2)node{$_{v_3}$};\draw(3.7,2.2)node{$_{v_4}$};\draw(5.4,2.2)node{$_{v_{|k|+1}}$};\draw(6.9,2.2)node{$_{v_{|k|+2}}$};\draw(1,-0.4)node{$_{u_1}$};\draw(2,-0.4)node{$_{u_2}$};\draw(3,-0.4)node{$_{u_3}$};\draw(4,-0.4)node{$_{u_4}$};\draw(7.4,-0.4)node{$_{u_{|k|+2}}$};\draw(6,-0.4)node{$_{u_{|k|+1}}$};
\draw(4,-1.5)node{ $_G$};
\end{tikzpicture}\vspace{1cm}
\begin{tikzpicture}[scale=0.6]

\fill(0,1) circle(3pt);
\fill(8,1) circle(3pt);
\draw (0,1)--(0,2);
\draw (8,1)--(8,2);

\draw(1,2)--(3,0);

\filldraw[fill=white](0,2) circle(3pt);
\filldraw[fill=white](8,2) circle(3pt);
\draw (4,0)--(1,0)--(0,1)--(1,2)--(4,2);
\draw[dotted](4,2)--(6,2);
\draw (6,2)--(7,2)--(8,1);\draw(7,0)--(6,0);
\draw[dotted](4,0)--(6,0);

\draw(0,2.4)node{$_{x'_1}$};\draw(8,2.4)node{$_{x'_2}$};
				\draw(5.7,3.4)node{$_{v'_{|k|+1}}$};\draw(7,3.4)node{$_{v'_{|k|+2}}$};
\foreach \nn in {1,2,3,4,6,7}
			{	\filldraw[fill=black] (\nn,2) circle (3pt);\draw(\nn,2)--(\nn,3);\draw (\nn,0)--(\nn,2);
				\filldraw[fill=white] (\nn,3) circle (3pt); 
				}
				\foreach \nn in {1,2,3,4}{
				\draw(\nn-0.3,3.4)node{$_{v'_{\nn}}$};}
				\foreach \nn in {3,4,6,7}
				\filldraw[fill=white] (\nn,0) circle (3pt);\filldraw[fill=white](1,0) circle(3pt);\filldraw[fill=white](2,0) circle(3pt);
\draw(0,0.6)node{$_{x_1}$};\draw(8,0.6)node{$_{x_2}$};
\draw(0.7,2.2)node{$_{v_1}$};\draw(1.7,2.2)node{$_{v_2}$};\draw(2.7,2.2)node{$_{v_3}$};\draw(3.7,2.2)node{$_{v_4}$};\draw(5.4,2.2)node{$_{v_{|k|+1}}$};\draw(6.9,2.2)node{$_{v_{|k|+2}}$};\draw(1,-0.4)node{$_{u_1}$};\draw(2,-0.4)node{$_{u_2}$};\draw(3,-0.4)node{$_{u_3}$};\draw(4,-0.4)node{$_{u_4}$};\draw(7.4,-0.4)node{$_{u_{|k|+2}}$};\draw(6,-0.4)node{$_{u_{|k|+1}}$};
\draw(4,-1.5)node{ $_{G'=G-u_{|k|+2}x_2}$};

\end{tikzpicture}
\end{center}
\caption{Construction of a graph $G$ such that $ \gwc (G-e)-\gwc (G)=k$ for some edge $e$, where $k$ is a negative integer. }\label{rys2}

\end{figure}
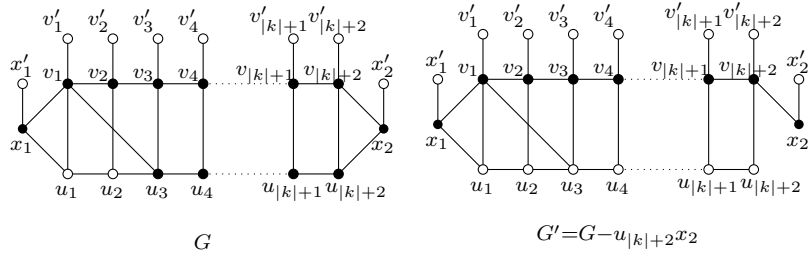


In \cite{lemanska2} is shown that removing an edge can not decrease the connected domination number, but it can increase it by at most two:
\begin{thm}{\rm \cite{lemanska2}}
If $e$ is an edge of $G$ and if $G$ and $G-e$ are connected, then $$\gc (G)\leq \gc (G-e)\leq \gc (G)+2.$$
\end{thm}

Hence, $\gc(G-e)-\gc(G)\in\{0,1,2\}$. 

We are going to study the influence of removing an edge over the weakly convex domination number when we consider   unicycle graphs or graphs such that every vertex is simplicial or cut-vertex.


\subsection{Graphs such that every vertex is simplicial or cut-vertex}

In this part, we consider graphs with at least three vertices   such that every vertex is simplicial or cut-vertex. Examples of such graphs are block graphs. We get another example  if we consider graphs $G$ with  $g(G)\geq 7$ and $\gwc (G)= \gamma_{c}(G)$  (see  Theorem~\ref{girth}).
 
Observe that if each vertex is a simplicial or a cut-vertex  in $G$ then  {by Corollary~\ref{cut_simplicial}} the unique minimum connected, and  also weakly convex, dominating set of $G$ is the set of all cut-vertices of $G$. 

\begin{lem}\label{lem1}{Let $G$ be a graph such that each vertex is either a cut-vertex or a simplicial vertex. If $e$ is not a cut edge of  $G,$ then $$\gamma_{c}(G)\leq \gamma_{c}(G-e)\leq \gamma_{c}(G)+1.$$}
\end{lem}
\begin{proof}
Let $D_0$ be a minimum connected dominating set of $G-e$. Then $D_0$ is also a connected dominating set of $G$ and $\gamma_c(G)\leq |D_0|=\gamma_c(G-e).$

Now let $D$ be a minimum connected dominating set of $G.$ Then $D=V_C.$ We consider two cases  for $e=uv$:\\
$Case$ $1.$ If $u,v\in V-D$, then $u,v$ are simplicial vertices. Then $D=V_C$ is a connected dominating set of $G-e$ and $\gamma_c(G-e)\leq |D|=\gamma_c(G)\leq \gamma_c(G)+1.$\\
$Case$ $2.$ If $|\{u,v\}\cap D|\geq 1$, then since $uv$ is not a cut edge, there is a vertex $v'\in N_G(v)\cap N_G(u).$ Then $D\cup \{v'\}$ is a connected dominating set  of $G-e$ and  $\gamma_c(G-e)\leq |D|+1= \gamma_c(G)+1.$

\end{proof}

\begin{lem}\label{lem3}   {Let $G$ be a graph such that each vertex is either a cut-vertex or a simplicial vertex. If $e$ is not a cut edge of $G$, then $\gamma_{c}(G-e)=\gwc(G-e).$}
\end{lem}
\begin{proof}
We consider two cases:\\
$Case$ $1.$ $e=uv,$ where both $u,v$ are simplicial vertices. Then $V_C$ is a minimum connected and weakly convex dominating set of $G-e$ and $|V_C|=\gamma_{c}(G-e)=\gwc(G-e).$\\
$Case$ $2.$ $|\{u,v\}\cap V_C|\geq 1$. Then, since $uv$ is not a cut edge, $|N_G(u)\cap N_G(v)|\geq 1$.
If  $(N_G(u)\cap N_G(v))\cap V_C\not =\emptyset$, then $V_C$ is a minimum connected and weakly convex dominating set of $G-e$ and $\gamma_c(G)=\gwc(G)=|V_C|.$ Otherwise, $V_C\cup \{v'\}$ is a minimum connected and weakly convex dominating set of $G-e$ and $\gamma_c(G)=\gwc(G)=|V_C|+1$ for any $v'\in N_G(u)\cap N_G(v)$.

\end{proof}

From  {Corollary~\ref{cut_simplicial}}, Lemma \ref{lem1} and Lemma \ref{lem3}, we have the following corollary.

\begin{wn}   {Let $G$ be a graph such that each vertex is either a cut-vertex or a simplicial vertex. If $e$ is not a cut edge of $G,$ then $$\gwc(G)\leq \gwc(G-e)\leq \gwc(G)+1.$$}
\end{wn}


\subsection{Unicyclic graphs}

\begin{obs}\label{obs1} Let $G$ be an unicyclic graph  with the only cycle $C_p$ and let $uv$ be any cycle edge. Thus 
\begin{enumerate}
\item  $G-uv$ is a   spanning tree $T$ of $G$; 
\item $V_L(G)\subseteq V_L(T)$; 
\item   $v\in V_L(T)-V_L(G)$  if and only if $d_G(v)=2.$
\end{enumerate}
\end{obs}

\begin{thm} \label{unicycle}Let $G$ be a connected unicyclic graph with the only cycle $C_p$.   If $e=uv$ is a cycle edge then $$\gwc(G) - 2 \leq  \gwc(G-e) \leq \gwc(G) + 2.$$
\end{thm}

\begin{proof}  
  Note that $G-e$ is a   spanning tree $T$ of $G.$ 
  We have that 
 $D_0=V-V_L(T)$ is  a minimum weakly convex dominating set of $T$ (Observation \ref{ob1}). 
Denote $a = \gwc(G).$
  We analyze the following three cases:

$Case$ $1.$  If $u,v \in D_0,$ then, since $u,v \notin  V_L(T)$, $d_G(u)\geq 3$  and $d_G(v)\geq 3.$ Thus $V_L(T)=V_L(G)$ and $D_0=V-V_L(G)$.

If $C=C_p$ with $p\geq 7$ then $V-V_L(G)$ is a minimum weakly convex dominating set of $G$ and   $\gwc(G)=|D_0|=\gwc(G-e).$ 

If  $p=4,5,6$ and there are $x,y$  two consecutive  vertices in $C$ with degree two then $V-(V_L(G)\cup  \{  x,y\})$ is a minimum weakly convex dominating set of $G,$ and $\gwc(G) = |D_0|-2=\gamma_{wcon}(T)-2$. 
In other case, for $p=5,6$, $D_0$ is a minimum  weakly convex dominating set of $G,$ so $\gwc(G)=|D_0|=\gwc(T).$ For $p=4$ if there exists a vertex in $C$ of degree $2$, then $\gwc(G)=|D_0| - 1=\gwc(T) - 1$, otherwise $\gwc(G)=|D_0|=\gwc(T)$.

If $p=3$, then denote the third vertex on the $C_3$ by $w$. Depending on $w\in V_C(G)$ or $w\not \in V_C(G)$, $D_0$ or $D_0-\{w\}$ is a minimum weakly convex dominating set of $G$.

  Thus, in this case, $\gwc(T)\in \{ a, a+1,a+2\}.$

$Case$ $2.$ Assume $|D_{0}\cap \{u,v\}|=1,$ without loss of generality let $u\in D_0, v\in V-D_0.$ Then $d_G(u)\geq 3$ and $d_G(v)=2.$ Note that $v\in V_L(T)-V_L(G).$ We have that $D_0=V-(V_L(G)\cup \{ v\})$ is a minimum weakly convex dominating set of $T.$ 

If $C=C_p$ with $p\geq 7$ then $V-V_L(G)$ is a minimum weakly convex dominating set of $G$  and   $\gwc(G)-1 = |D_0| =\gwc(T).$

If  $p=4,5,6$ and there are $x,y$  two consecutive  vertices   in $C$ with degree two then $V-(V_L(G)\cup  \{  x,y\})$ is a minimum weakly convex dominating set of $G$ 
and $\gwc(G) +1= |D_0|=\gwc(T).$
In other case, for $p=5,6$, $V-V_L(G) $ is a minimum weakly convex dominating set of $G$, i.e., $\gwc(G)-1= \gwc(T)$, and for $p=4$, $D_0$ is a minimum weakly convex dominating set of $G$, i.e.  $\gwc(G)= \gwc(T)$.

If $p=3$, then denote the third vertex on the $C_3$ by $w$. Depending on $w\in V_C(G)$ or $w\not \in V_C(G)$,  $D_0$ or $D_0-\{w\}$ is a minimum weakly convex dominating set of $G$.

   Thus, in this case, $\gwc(T)\in \{ a-1, a, a+1\}.$  

 $Case$ $3.$ Now let $u,v \in V-D_0.$ Then $d_G(u) =2=d_G(v).$ Note that $u,v\in V_L(T)-V_L(G).$ We have that $D_0=V-(V_L(G)\cup \{ u,v\})$ is a minimum weakly convex dominating set of $T.$ 

If $C=C_p$ with $p\geq 7$, then $V-V_L(G)$ is a minimum weakly convex dominating set of $G$  and   $\gwc(G)-2 = |D_0| =\gwc(T).$
If  $p=4,5,6$, then,  since there are  two consecutive  vertices $u,v$  in $C$  with degree two, $V-(V_L(G)\cup  \{ u,v\})$ is also a minimum weakly convex dominating set of $G$ 
and
 $\gwc(G)  =  \gwc(T).$
 
   Thus, in this case, $\gwc(T)\in \{ a, a-2\}.$ 
 
 We conclude that,   $\gwc(T) \in \{a-2, a-1, a, a+1, a+2\}$ for every spanning tree $T.$
\end{proof}

Let $\mathcal{T}$ be the set of all spanning trees of $G.$ We say that function $\Pi$ with
integer values {\em interpolates over graph} $G$ if and only if $\Pi(\mathcal{T}(G))$ is an interval; i.e the set $\Pi(\mathcal{T}(G)) = \{ \Pi(T) : T \in \mathcal{T}(G)\}$ consists of consecutive integers. 
Function $\Pi$ is an {\em interpolating function} if $\Pi$ interpolates over every connected
graph.
The interpolation properties of some domination parameters were presented
in \cite{kuziak} and
in particular,  it  was shown that $\gamma_c$ is an interpolating function.

In \cite{topp} Topp and Vestergaard proved the following theorem.
\begin{thm}\label{topp_vest} \cite{topp}
  Function $\Pi$ with integer values is an interpolating function if and
only if $\Pi$ interpolates over every unicyclic graph.
\end{thm}

Now, we are in condition to   show that weakly convex domination number is also an interpolating function.

\begin{wn} $\gwc$ is an interpolating function.
\end{wn}
\begin{proof}
Using Theorem \ref{unicycle} we have that $\gwc(T) \in \{a-2, a-1, a, a+1, a+2\}$ for every spanning tree $T,$ so by Theorem \ref{topp_vest} we have that $\gwc$ is an interpolation function.
 \end{proof}

  \section{Acknowledgment} 
 The third author thankfully acknowledge support
by National Science Centre (Poland) grant number 2015/17/B/ST6/01887.  The fourth author thankfully acknowledge support
from       TIN2017-85160-C2-1-R from MIMECO of Spain.

\end{document}